\documentclass[12pt,oneside,reqno]{amsart}
\usepackage{mathrsfs, amsmath,amssymb}
\usepackage{fullpage}
\numberwithin{equation}{section}
\usepackage{graphicx}
\usepackage{svg}
\newtheorem{teo}{Theorem}[section]

\newtheorem{pro}[teo]{Proposition}
\newtheorem{lem}[teo]{Lemma}
\newtheorem{cor}[teo]{Corollary}
\newtheorem{rem}[teo]{Remark}

\def\k{\kappa}

\newcommand{\bmultg}{\!\begin{array}{c} {\scriptstyle\times}
    \\[-12pt]\cup\end{array}\!}

\title{Spectral distribution of the free Jacobi process, revisited}

\author[T. Hamdi]{Tarek Hamdi}
\address{Department of Management Information Systems \\ College of Business Administration\\ Qassim University \\ Saudi Arabia 
and Laboratoire d'Analyse Math\'ematiques et applications \\ LR11ES11 \\ Universit\'e de Tunis El-Manar \\ Tunisie.
}
\email{tarek.hamdi@mail.com}

\begin{document}
\begin{abstract}
We obtain a description for the spectral distribution of the free Jacobi process for any initial pair of projections. This result relies on a study of the unitary operator $RU_tSU_t^*$ where
 $R,S$ are two symmetries and $U_t$ a free unitary Brownian motion, freely independent from $\{R,S\}$.  In particular, for non-null traces of $R$ and $S$, we prove that the spectral measure of $RU_tSU_t^*$ possesses two atoms at $\pm1$ and an $L^\infty$-density on the unit circle $\mathbb{T}$, for every $t>0$. Next, via a  Szeg\H{o} type transform of this law, we obtain a full description of the spectral distribution of $PU_tQU_t^*$ beyond the $\tau(P)=\tau(Q)=1/2$ case. Finally, we give some specializations for which these measures are explicitly computed.
\end{abstract}
\maketitle

\section{Introduction}
Let $P,Q $ be two projections  in a $W^*$-probability space  $(\mathscr{A},\tau)$ which are free with $\{U_t,U_t^*\}$. 
The present paper is a companion to the series of papers \cite{Col-Kem,Dem,Demni,Dem-Ham,DHH,Dem-Hmi} devoted to the study of the spectral distribution, hereafter $\mu_t$, of the self-adjoint-valued process $(X_t:=PU_tQU_t^*P)_{t\geq0}$. 
Viewed in the compressed algebra  $(P\mathscr{A}P,\tau/\tau(P))$, $X_t$ coincide with the so-called free Jacobi process with parameter $\left(\tau(P)/\tau(Q),\tau(Q)\right)$, introduced by Demni in \cite{Dem} via free stochastic calculus, as solution to a free SDE there.
Properties of its measure play important roles in free entropy and free information theory (see e.g. \cite{Tar,Tarek,Hia-Ued,Izu-Ued,Voi}).
Furthermore, $\mu_t$ completely determines the structure of the von Neumann algebra generated by $P$ and $U_tQU_t^*$ (see e.g. \cite{Hia-Ued,Rae-Sin}) for any $t\geq0$, yielding a continuous interpolation from the law of $PQP$ (when $t=0$) to the free multiplicative convolution of the spectral measures of $P$ and $Q$ separately (when $t$ tends to infinity). Indeed,
 the pair $(P,U_tQU_t^*)$ tends towards $(P,UQU^*)$ as $t\rightarrow\infty$, where $U$ is a Haar unitary free from $\{P,Q\}$.  The two projections $P$ and $UQU^*$ are therefore free (see \cite{Nic-Spe}) and hence  $\mu_{PUQU^*P}=\mu_P\boxtimes\mu_{UQU^*}=\mu_P\boxtimes\mu_Q$. This measure was explicitly computed in \cite[Example 3.6.7]{Dyk-Nic-Voi}.
Generally, the operators $P$ and $U_tQU_t^*$ are not free for finite $t$ and the process $t\mapsto(P,U_tQU_t^*)$ is known as the free liberation of the pair $(P,Q)$ (cf. \cite{Voi}). 
When both projections coincide, the series of papers \cite{Demni,Dem-Ham,DHH,Dem-Hmi} aim to determine $\mu_t$ for any $t>0$. 
In particular, when $P=Q$ and $\tau(P)=1/2$, Demni, Hmidi and myself proved in \cite[Corollary 3.3]{DHH} that the measure $\mu_t$ possesses a continuous density on $(0,1)$ for $t>0$ which fits that of the random variable $(I+U_{2t}+(I+U_{2t})^*)/4$.
In \cite{Col-Kem}, Collins and Kemp extended this result to the case of two projections $P,Q$ with traces 1/2. 
Afterwards this result was partially extended by Izumi and Ueda to the arbitrary traces case. They proved the following.
\begin{align*}
\mu_t=(1-\min\{\tau(P),\tau(Q)\})\delta_{0}+\max\{\tau(P)+\tau(Q)-1,0\}\delta_1+\gamma_t
\end{align*}
where  $\gamma_t$ is a positive measure with no atom on $(0,1)$ for every $t>0$ (cf. \cite[Proposition 3.1]{Izu-Ued}). When $\tau(P)=\tau(Q)=1/2$, this measure coincide with the Szeg\H{o} transformation of the distribution of $UU_t$ where $U$ is a unitary random variable determined by the law of $PQP$ (cf. \cite[Proposition 3.3]{Izu-Ued}).
In \cite[Lemma 3.2, Lemma 3.6]{Col-Kem}, Collins and Kemp studied the support of the measure $\gamma_t$, for the general case of traces, and the way in which the edges of this support are propagated, but they were still not able to prove the continuity of $\gamma_t$.

Our major result in these notes is a complete analysis of the spectral distribution of  the unitary operator $RU_tSU_t^*$ (hereafter $\nu_t$) for any symmetries $R,S\in \mathscr{A}$ which are free with $\{U_t,U_t^*\}$.  In particular, we prove that the measure
\begin{align*}
\nu_t-\frac{1}{2}\left|\tau(R)-\tau(S)\right|\delta_\pi-\frac{1}{2}|\tau(R)+\tau(S)|\delta_0
\end{align*}
possesses an $L^\infty$-density $\kappa_t$ on $\mathbb{T}=(-\pi,\pi]$. Using the relationship between $\mu_t$ and $\nu_t$, when $\{P,Q\}$ and $\{R,S\}$ are associated (cf. \cite[Theorem 4.3]{Tar}), we deduce the regularity of 
$\mu_t$ for any initial projections. In particular, we prove that the measure $\gamma_t$ possesses a continuous density on $[0,1]$. Here is our result.

 \begin{teo}\label{description mu}
 Let $P,Q$ be orthogonal projections and  $U_t$ a free unitary Brownian motion, freely independent from $P,Q$.
For every $t>0$, the spectral distribution $\mu_t$ of the self adjoint operator $PU_tQU_t^*P$ is given by
\begin{align*}
\mu_t=(1-\min\{\tau(P),\tau(Q)\})\delta_{0}+\max\{\tau(P)+\tau(Q)-1,0\}\delta_1+\frac{\kappa_t(2\arccos(\sqrt{x}))}{2\pi\sqrt{x(1-x)}}{\bf 1}_{[0,1]}(x)dx.
\end{align*}
\end{teo}
We conclude the paper with the following `unexpected' identities for the measure $\nu_t$ when the initial operators are assumed to be freely, classically, boolean and monotone independent with law $\frac{\delta_1+\delta_{-1}}{2}$. We have  $\nu_t$ is constant in $t$ in the first case, and its given by a dilation of the law of $U_t$ in the rest of cases. The result is as follows.
\begin{teo}\label{convol}
Let $\lambda_t$ be the probability distribution of the free unitary Brownian motion $U_{t}$ and $\mu=\frac{\delta_1+\delta_{-1}}{2}$ (considered as a law on $\mathbb{T}$). We denote respectively by $\boxtimes,*,\bmultg$ and $\triangleright$  the free, classical, boolean and monotone multiplicative convolutions.
Then, for all $t\geq0$,
\begin{enumerate}
\item The measure $\big(\mu\boxtimes\mu\big)\boxtimes\lambda_{t}$  coincide with $\mu\boxtimes\mu$.
\item The push-forward of $\big(\mu*\mu\big)\boxtimes\lambda_{t}$ by the map $z\mapsto z^2$ coincide with the law of $U_{2t}$.
\item The push-forward of $\big(\mu\bmultg\mu\big)\boxtimes\lambda_{t}$ by the map $z\mapsto z^3$ coincide with the law of $U_{3t}$.
\item The push-forward of $\big(\mu\triangleright\mu\big)\boxtimes\lambda_{t}$ by the map $z\mapsto z^4$ coincide with the law of $U_{4t}$.
\end{enumerate}
 \end{teo} 

The paper is organized as follows. 
We start in Section 2 with some preliminaries which gathers useful information about  the Herglotz transform of probability measures on the unit circle, and the spectral distribution of the free unitary Brownian motion.
Section 3 fixes the basic ideas and notations for the rest of the work presented.
Section 4 deals with regularity properties of the spectral measure $\nu_t$ and  gives a proof of the Theorem \ref{description mu}.
Section 5 consists in explicit computations of densities in certain special cases for initial operators.
 
 \section{Preliminaries}
This section gives a concise review about some ideas we will use to prove our main results.
 
 \subsection{The Herglotz transform}
 Let $\mathscr{M}_\mathbb{T}$ denotes the set of probability measures on the unit circle $\mathbb{T}$.  The normalized Lebesgue measure on  $\mathbb{T}$ will be denoted $m$.
 The Herglotz transform $H_\mu$ of a measure $\mu\in \mathscr{M}_\mathbb{T}$ is the analytic function in the unit disc $\mathbb{D}$ defined by the formula
 \begin{equation*}
H_\mu(z)=\int_{\mathbb{T}}\frac{\zeta+z}{\zeta-z}d\mu(\zeta).
\end{equation*}
This function is related to the moments generating function of the measure $\mu$
\begin{equation*}
\psi_{\mu}(z)=\int_{\mathbb{T}}\frac{z}{\zeta -z}d\mu(\zeta), \quad z\in \mathbb{D}
\end{equation*}
by the simple formula $H_{\mu}(z)=1+2\psi_{\mu}(z)$. Since any distribution on the unit circle is uniquely determined by its moments, we deduce that $H_\mu$ determines uniquely $\mu$.
One of its major importance  is due to the following result (see e.g. \cite[ Theorem 1.8.9]{CMR}):
\begin{teo}[Herglotz]\label{Herglotz}
 The Herglotz transform sets up a bijection between the analytic functions $H$ on $\mathbb{D}$ with $\Re H\geq0$ and $H(0)>0$ and the non-zero measures $\mu\in \mathscr{M}_\mathbb{T}$.
\end{teo}
For $0< p<\infty$, let $H^p(\mathbb{D})$ be the space of analytic functions $f$ on $\mathbb{D}$ such that
\begin{align*}
\sup_{0<r<1}\int_{\mathbb{T}}|f(r\zeta)|^pd\zeta<\infty.
\end{align*}
For $p = \infty$, let $ H^\infty(\mathbb{D})$ denote the Hardy space consisting of all bounded analytic functions on $\mathbb{D}$ with the sup-norm.
Let $L^p(\mathbb{T})$ denote the Lebesgue spaces on the circle $\mathbb{T}$ with respect to the normalized Lebesgue measure. The following result proves the existence of a boundary function for all $f\in H^p(\mathbb{D})$ (see \cite[Theorem 1.9.4]{CMR}).
\begin{teo}[\cite{CMR}]
Let $0< p\leq\infty$ and $f\in H^p(\mathbb{D})$, the boundary function $\tilde{f}(\zeta)$ exists for $m$-almost all $\zeta$ in $\mathbb{T}$ and belongs to $L^p(\mathbb{T})$. Furthermore, the norms of $f$ in $H^p(\mathbb{D})$ and of $\tilde{f}(\zeta)$ in $L^p(\mathbb{T})$ coincide.
\end{teo}
We know (see e.g.  \cite[Lemma 2.1.11 ]{CMR}) that $H_\mu\in H^p(\mathbb{D})$ for all $0<p<1$, then $\widetilde{H}_\mu(\zeta)$ exists for $m$-almost all $\zeta$ in $\mathbb{T}$. 
The density of $\mu$ can be recovered then from the boundary values of $\Re H_\mu$ 
by Fatou's theorem (\cite[Theorem 1.8.6]{CMR}) since  $\Re \widetilde{H}_\mu=d\mu/dm$ $m$-a.e.
Note that the atoms of $\mu\in \mathscr{M}_\mathbb{T}$ can also be recovered from $H_\mu$ by Lebesgue’s dominated convergence theorem, 
via
 \begin{equation*}
\lim_{r\rightarrow1^-} (1-r)H_{\mu}(r\zeta)=2\mu\{\zeta\}\quad {\rm for\ all}\ \zeta\in\mathbb{T}.
 \end{equation*}
\subsection{Spectral distribution of the free unitary Brownian motion}
For $\mu\in \mathscr{M}_\mathbb{T}$, let $\psi_\mu$ denote its moments generating function and $\chi_\mu$ the function $\frac{\psi_\mu}{1+\psi_\mu}$. If $\mu$ has nonzero mean, we denote by $\chi_\mu^{-1}$ the inverse function of $\chi_\mu$ in some neighborhood of zero. In this case the $\Sigma$-transform of $\mu$ is defined by  $\Sigma_\mu(z)=\frac{1}{z}\chi_\mu^{-1}(z)$.
The distribution $\lambda_t$ of the free unitary Brownian motion was introduced by Biane in \cite{Biane} as the unique probability measure on $\mathbb{T}$ such that its $\Sigma$-transform is given by
\begin{align*}
\Sigma_{\lambda_t} (z)=\exp\left(\frac{t}{2}\frac{1+z}{1-z}\right).
\end{align*}
This measure $ \lambda_t$ is known as the multiplicative analogues of semicircular distributions. Its moments
follow from the large-size asymptotic of observables of the free Brownian motion (of dimension $d$) $(U_{t}^{(d)})_{t\geq0}$ on the unitary group $\mathscr{U}(d)$ as follows.
\begin{align*}
\lim_{d\rightarrow \infty}\frac{1}{d}\mathbb{E}\left( \textrm{tr}[U_{t/d}^{(d)}]^{k}\right)=\int_{\mathbb{T}}\zeta^kd\lambda_t(\zeta),\quad k\geq0.
\end{align*}
This result was proved independently by Biane and Rains in \cite{Biane, Rains}  where these moments are explicitly calculated:
\begin{align}\label{mom}
\tau\big(U_t^k\big)=e^{-kt/2}\sum_{j=0}^{k-1}\frac{(-t)^j}{j!}\binom{k}{j+1}k^{j-1},\quad k\geq0.
\end{align}
The equality \eqref{mom} can be transformed into the PDE
\begin{equation}\label{standardpde}
\partial_tH+zH\partial_zH= 0,
\end{equation}
with the initial condition $H(0,z)=(1+z)/(1-z)$ for the Herglotz transform $H_{\lambda_{2t}}(z)$ 
(see e.g. the proof of \cite[Proposition 3.3]{Izu-Ued}).
 The measure $\lambda_t$ is described in \cite{Biane1} from the boundary behaviour of the inverse function of  $H_{\lambda_{t}}(z)$ as follows.
\begin{teo}[\cite{Biane1}] \label{bian}
For every $t > 0$, $\lambda_t$ has a continuous density $\rho_t$ with respect to the normalized Lebesgue
measure on $\mathbb{T}$. Its support is the connected arc $\{e^{i\theta}: |\theta|\leq g(t)\}$ with
\begin{align*}
g(t):= \frac{1}{2}\sqrt{t(4-t)}+\arccos\big(1-\frac{t}{2}\big) 
\end{align*}
for $t\in[0, 4]$, and the whole circle for $t > 4$. The density $\rho_t$ is determined by $\Re h_t(e^{i\theta})$  where $z = h_t(e^{i\theta})$  is the unique solution (with positive real part) to
\begin{align*}
\frac{z-1}{z+1}e^{\frac{t}{2}z}=e^{i\theta}.
\end{align*}
\end{teo}

 \section{Reminder and notations}
 We use here the same symbols as in \cite{Tar,Tarek}. 
To a given pair of projections $P,Q$ in $\mathscr{A}$ that are independent of $(U_t)_{t\geq0}$ we associate the symmetries $R=2P-I$ and $S=2Q-I$. Denote by $\alpha=\tau(R)$ and $\beta=\tau(S)$. We sometimes use the notations $a=|\alpha-\beta|/2$ and $b=|\alpha+\beta|/2$ for simplicity.
Keep the symbols $\mu_t$ and $\nu_t$ above. The unit circle is identified with $(-\pi,\pi]$ by $e^{i\theta}$. According to \cite[Theorem 4.3]{Tar}, the measure $\nu_t$ is  connected to $\mu_t$ by the following formula
 \begin{equation*}
\nu_t=2\hat{\mu}_t-\frac{2-\alpha-\beta}{2}\delta_{\pi}-\frac{\alpha+\beta}{2}\delta_0,
\end{equation*}
where 
 \begin{equation}\label{symm}
\hat{\mu}_t:=\frac{1}{2}\left(\tilde{\mu}_t+\left(\tilde{\mu}_t|_{(0,\pi)}\right)\circ j^{-1}\right)
\end{equation}
 is the symmetrization on $(-\pi,\pi)$, with the mapping $j:\theta\in(0,\pi)\mapsto-\theta\in(-\pi,0)$, of the positive measure $\tilde{\mu}_t(d\theta)$ on $[0,\pi]$ obtained from $\mu_t(dx)$ via the variable change $x=\cos^2(\theta/2)$.
 Equivalently, we obtain the following relationship between the Herglotz transforms $H_{\mu_t}$ and $H_{\nu_t}$ (see \cite[Corollary 4.2]{Tar}).
\begin{equation}\label{relationship}
H_{\nu_t}(z)=\frac{z-1}{z+1}H_{\mu_t}\left(\frac{4z}{(1+z)^2}\right)-2(\alpha+\beta)\frac{z}{z^2-1}.
\end{equation}
The function $H_{\nu_t}(z)$, which we shall denote by $H(t,z)$, is analytic in both variables $z\in\mathbb{D}$ and $t>0$ (see  \cite[Theorem 1.4]{Col-Kem}) and solves the PDE (see \cite[Proposition 2.3]{Tar})
 \begin{equation}\label{pde}
\partial_tH+zH\partial_zH= \frac{2 z \left(\alpha z^2+2 \beta z+\alpha\right) \left(\beta z^2+2  \alpha z+\beta\right)}{\left(1-z^2\right)^3}.
\end{equation}
Let
\begin{equation}\label{defK}
K(t,z):=\sqrt{H(t,z)^2-\left(a\frac{1-z}{1+z}+b\frac{1+z}{1-z} \right)^2}.
\end{equation}
The PDE \eqref{pde} is then transformed into
\begin{align*}
\partial_tK+zH(t,z)\partial_zK=0.
\end{align*}
Note that steady state solution $K(\infty,z)$ is the constant $\sqrt{1-(a+b)^2}$ (see \cite[Remark 3.3]{Tar}). 
The ordinary differential equations (ODEs for short) of characteristic curve associated with this PDE are as follows.
\begin{align}\label{characteristic}
\begin{cases}
\partial_t\phi_{t}(z) = \phi_{t}(z) H(t,\phi_{t}(z)), \quad \phi_{0}(z) = z,\\
\partial_t \left[K(t,\phi_{t}(z))\right] =0
\end{cases}
\end{align}
The second ODE of \eqref{characteristic} implies that $K(t,\phi_{t}(z))=K(0,z)$, while the first one is nothing else but the radial Loewner ODE  (see  \cite[Theorem 4.14]{Law}) which defines a unique family of conformal transformations $\phi_t$ from some region $\Omega_{t}\subset\mathbb{D}$ onto $\mathbb{D}$ with $\phi_t(0)=0$ and $\partial_z\phi_t(0)=e^t$. Moreover, from \cite[Remark 4.15]{Law}, $\phi_t$ is invertible from $\Omega_{t}$ onto $\mathbb{D}$ and it has a continuous extension to $\mathbb{T}\cap \overline{\Omega_{t}}$ by \cite[Proposition 2.1]{Tarek}.
Integrating the first ODE in \eqref{characteristic}, we get
\begin{align*}
\phi_t(z)=z\exp\left(\int_0^tH(s,\phi_s(z))ds\right).
\end{align*}
Let us define
\begin{align*}
h_t(r,\theta)=1-\int_0^t\frac{1-|\phi_s(re^{i\theta})|^2}{-\ln r}\int_{\mathbb{T}}\frac{1}{|\xi-\phi_s(re^{i\theta})|^2}d\nu_s(\xi)ds,
\end{align*}
so that
\begin{align}\label{logmod}
\ln \left|\phi_t(re^{i\theta})\right|=\ln r+\Re\int_0^tH(s,\phi_s(re^{i\theta}))ds=(\ln r) h_t(r,\theta).
\end{align}
Define $R_t:[-\pi,\pi]\rightarrow [0,1]$ as follows
\begin{align*}
R_t(\theta)&=\sup\left\{r\in(0,1):h_t(r,\theta)>0\right\},
\end{align*}
and let
\begin{align*}
I_{t}&=\left\{\theta\in[-\pi,\pi]:h_t(\theta)<0\right\}
\end{align*}
where $h_t(\theta)=\lim_{r\rightarrow 1^-}h_t(r,\theta)\in\mathbb{R}\cup\{-\infty\}$ (see the fact exposed under Lemma 3.2 in \cite{Tarek}).
The next result, giving a description of $\Omega_{t}$  and its boundary, was proved in \cite[Proposition 3.3]{Tarek}. 
\begin{pro}[\cite{Tarek}]
For any $t>0$, we have
\begin{enumerate}
\item $\Omega_{t}=\{re^{i\theta}: h_t(r,e^{i\theta})>0\}$ 
\item $\partial\Omega_{t}\cap\mathbb{D}=\{re^{i\theta}: h_t(r,e^{i\theta})=0\ {\rm and}\ \theta\in I_{t}\}$.
\item $\partial\Omega_{t}\cap\mathbb{T}=\{e^{i\theta}: h_t(r,e^{i\theta})=0\ {\rm and}\ \theta\in [-\pi,\pi]\setminus I_{t}\}$.
\end{enumerate}
\end{pro}
In closing, we recall the following result which will be of use later on (see the proof of Theorem 1.1 in \cite{Tarek}).
\begin{lem}[\cite{Tarek}]\label{ext}
For every $t>0$, the function $K(t,.)$ has a continuous extension to the unit circle $\mathbb{T}$.
\end{lem}

     \section{Analysis of spectral distributions of $RU_tSU_t^*$}
     In this section, we shall prove the Theorem \ref{description mu}. To this end, we start by giving a description of the spectral measure $\nu_t$ of $RU_tSU_t^*$ for any  $t>0$, and deriving a formula for its density. 
We notice that from the asymptotic freeness of $R$ and $U_tSU_t^*$, the measure $\nu_t$ converges weakly as $t\rightarrow\infty$ (see \cite[Proposition 2.6]{Tar}) to
\begin{equation}\label{state}
\nu_\infty=a\delta_{\pi}+b\delta_0+\frac{\sqrt{-(\cos\theta-r_+)(\cos\theta-r_-)}}{2\pi|\sin\theta|}{\bf 1}_{(\theta_-,\theta_+)\cup(-\theta_+,-\theta_-)}d\theta
\end{equation}
with
$ r_{\pm}=-\alpha\beta\pm\sqrt{(1-\alpha^2)(1-\beta^2)}\quad {\rm and}\ \theta_{\pm}=\arccos r_{\pm}.
$
The following theorem asserts that an analogous result holds for finite $t$.
  \begin{teo}\label{descrip}
For every  $t>0$,  $\nu_t-a\delta_{\pi}-b\delta_0$  is absolutely continuous with respect to the normalized Lebesgue measure on $\mathbb{T}=(-\pi,\pi]$.
Moreover, its  density  $\kappa_t$ at the point $e^{i\theta}$ is equal to the real part of
\begin{equation*}
\sqrt{[K(t,e^{i\theta})]^2+(a+b)^2-1-\frac{(\cos\theta-r_+)(\cos\theta-r_-)}{\sin^2\theta}}.
\end{equation*}
\end{teo}
\begin{proof}
Define the function
\begin{align*}
L(t,z)&=\int_{\mathbb{T}}\frac{e^{i\theta}+z}{e^{i\theta}-z}(\nu_t-a\delta_\pi-b\delta_0)(d\theta)
\\&=H(t,z)-a\frac{1-z}{1+z}-b\frac{1+z}{1-z}.
\end{align*}
The real part of this function is nothing else but  the Poisson integral of the measure $\nu_t-a\delta_\pi-b\delta_0$. 
Using \eqref{defK} and multiplying by the conjugate, we get
\begin{align*}
L(t,z)
&= \frac{K(t,z)^2}{\sqrt{K(t,z)^2+\left(a\frac{1-z}{1+z}+b\frac{1+z}{1-z}\right)^2}+a\frac{1-z}{1+z}+b\frac{1+z}{1-z}}
\\&=\frac{(1-z^2)K(t,z)^2}{\sqrt{[(1-z^2)K(t,z)]^2+\left[a(1-z)^2+b(1+z)^2\right]^2}+a(1-z)^2+b(1+z)^2}.
\end{align*}
Note that  $K(t,z)$ extends continuously to $\mathbb{T}$ by Lemma \ref{ext}. 
The denominator of the above expression does not vanish on the closed unit disc and
\begin{align*}
z\mapsto (1-z^2)^2K(t,z)^2+\left[a(1-z)^2+b(1+z)^2\right]^2=(1-z^2)H(t,z)^2
\end{align*}
does not take negative values.
These  together imply that $L(t,z)$ has a continuous extension on the boundary $\mathbb{T}$.
Hence, by uniqueness of Herglotz representation (see Theorem \ref{Herglotz}), the measure $\nu_{t}-a\delta_\pi -b\delta_0$ is absolutely continuous with respect to the Haar measure in $\mathbb{T}$ and its density is given by:
\begin{align*}
\Re\left[H(t,e^{i\theta})-a\frac{1-e^{i\theta}}{1+e^{i\theta}}-b \frac{1+e^{i\theta}}{1-e^{i\theta}}\right]=&\Re\sqrt{[K(t,e^{i\theta})]^2+\left[  a\frac{1-e^{i\theta}}{1+e^{i\theta}}-b \frac{1+e^{i\theta}}{1-e^{i\theta}}\right]^2}
\\=&\Re\sqrt{[K(t,e^{i\theta})]^2-\left[a\tan(\theta/2)-b\cot(\theta/2)\right]^2}.
\end{align*}
To complete the proof, we need only show that
\begin{align*}
\left[a\tan(\theta/2)-b\cot(\theta/2)\right]^2=1-(a+b)^2+\frac{(\cos\theta-r_+)(\cos\theta-r_-)}{\sin^2\theta}
\end{align*}
or equivalently that
\begin{align*}
(1-a^2-b^2)\sin^2\theta-a^2\sin^2\theta\tan^2(\theta/2)-b^2\sin^2\theta\cot^2(\theta/2)=-(\cos\theta-r_+)(\cos\theta-r_-).
\end{align*}
Working from the left-hand side and using the identities
\begin{align*}
\sin^2\theta=1-\cos^2\theta,\quad \sin^2\theta\tan^2(\theta/2)=(1-\cos\theta)^2,\quad\sin^2\theta\cot^2(\theta/2)= (1+\cos\theta)^2,
\end{align*}
we get
\begin{align*}
(1-a^2-b^2)(1-\cos^2\theta)-a^2(1-\cos\theta)^2-b^2(1+\cos\theta)^2.
\end{align*}
Rearranging these terms, we obtain
\begin{align*}
-\cos^2\theta+2(a^2-b^2)\cos\theta-2(a^2+b^2)+1.
\end{align*}
So, by substituting the equalities $\alpha\beta=b^2-a^2$ and $\alpha^2+\beta^2=2(a^2+b^2)$, we obtain the required formula:
\begin{align*}
-\cos^2\theta-2\alpha\beta\cos\theta+1-\alpha^2-\beta^2=-(\cos\theta-r_+)(\cos\theta-r_-).
\end{align*}
\end{proof}

\begin{pro}
The support of $\nu_t$ is a subset of $\{\phi_t\left( R_t(\theta)e^{i\theta}\right): \theta\in I_t\}$.
\end{pro}
\begin{proof}
By \eqref{logmod}, we have
\begin{align*}
\int_0^t\Re H\left(s,\phi_s\left( R_t(\theta)e^{i\theta}\right)\right)ds=-\ln R_t(\theta)
\end{align*}
where we used the fact that $ \ln\left|\phi_t\left( R_t(\theta)e^{i\theta}\right)\right|=0$ due to the equality $\left|\phi_t\left( R_t(\theta)e^{i\theta}\right)\right|=1$.
Then, by continuity of $s\mapsto \Re H\left(s,\phi_s\left( R_t(\theta)e^{i\theta}\right)\right)$ on $[0,t]$, we deduce that the assertion 
$\Re H\left(t,\phi_t\left( R_t(\theta)e^{i\theta}\right)\right)>0$ yields $R_t(\theta)\neq1$. Finally, by definition of $R_t(\theta)$ and $I_t$, we have
\begin{align*}
\{\theta:R_t(\theta)\neq1\} =&\{\theta: \exists\ r_0 \in(0,1),\ h_t(r_0,e^{i\theta})=0\}
\\=&\{\theta:h_t(\theta)<0\}
\\=&I_t.
\end{align*}
\end{proof}

\begin{pro}
The density  $\kappa_t$ of $\nu_t-a\delta_{\pi}-b\delta_0$ belongs to $L^\infty(\mathbb{T})$.
\end{pro}
\begin{proof}
By \eqref{defK}, we have
\begin{align*}
K(t,z)^2&=H(t,z)^2-\left(a\frac{1-z}{1+z}+b\frac{1+z}{1-z} \right)^2
\\&=L(t,z)\left(L(t,z)+2a\frac{1-z}{1+z}+2b\frac{1+z}{1-z}\right).
\end{align*}
Then
\begin{align*}
(\Re L(t,z))^2\leq\Re L(t,z)\Re\left(L(t,z)+2a\frac{1-z}{1+z}+2b\frac{1+z}{1-z}\right)\leq |K(t,z)^2|.
\end{align*}
Since the function $K(t,z)$ is analytic in $\mathbb{D}$ and extends continuously to $\mathbb{T}$, it becomes then
 of Hardy class $H^{\infty}(\mathbb{D})$, and hence the density of $\nu_t-a\delta_\pi-b\delta_0$ belongs to $L^\infty(\mathbb{T})$ by \cite[Theorem p. 15]{Koo}.
\end{proof}
We now proceed to the proof of Theorem \ref{description mu}.
 
\begin{proof}[ Proof of Theorem \ref{description mu}]

Using Theorem \ref{descrip} and in \cite{Tar}, we have 
\begin{align*}
\nu_t-a\delta_{\pi}-b\delta_0=2[\hat{\mu_t}-(1-\min\{\tau(P),\tau(Q)\})\delta_{\pi}-\max\{\tau(P)+\tau(Q)-1,0\}\delta_0].
\end{align*}
From  Theorem 4.3, this measure is absolutely continuous with respect to the normalized Lebesgue measure $d\theta/2\pi$  on $\mathbb{T}=(-\pi,\pi]$ with density the function $\kappa_t$. Hence, by \eqref{symm}, we have 
\begin{align*}
(\tilde{\mu_t}-(1-\min\{\tau(P),\tau(Q)\})\delta_{\pi}-\max\{\tau(P)+\tau(Q)-1,0\}\delta_0)(d\theta)= \kappa_t(\theta)\frac{d\theta}{2\pi}, \quad \theta\in [0,\pi]
\end{align*}
and so the desired assertion holds via the variable change $\theta=2\arccos(\sqrt{x})$.
\end{proof}



\section{Special cases}
We present here some specializations for which the measure  $\nu_t$ (and hence $\mu_t$) is explicitly determined. 
\subsection{Centered initial operators}
 i.e. $\tau(R)=\tau(S)=0$ or $a=b=0$.
In this case, the PDE \eqref{pde} rewrites
 \begin{equation*}
\partial_tH+zH\partial_zH= 0,
\end{equation*}
and the measure $\nu_{t}$  becomes identical to the probability distribution of  $UU_{2t}$ where $U$ is a free unitary whose distribution is $\nu_0$ (see \cite[Proposition 3.3]{Izu-Ued} or \cite[Remark 4.7]{Tar}). Hence, the measure $\nu_t$ is given by the multiplicative free convolution $\nu_0\boxtimes\lambda_{2t}$, studied by Zhong in \cite{Zho}.  The density of this measure and its support are explicitly computed in \cite[Theorem 3.8 and Corollary 3.9]{Zho}.
In particular, when $\nu_0$ is a Dirac mass at 1 (on the unit circle), the Herglotz transforms $H(t,z)$ of $\nu_t$ satisfy the PDE 
 \begin{equation*}
\partial_tH+zH\partial_zH= 0,\quad H(0,z)=\frac{1+z}{1-z}.
\end{equation*}
Then it follows from uniqueness of solution of \eqref{standardpde} that $H(t,z)=H_{\lambda_{2t}}(z)$, and by uniqueness of Herglotz representation, $\nu_t$ coincide with the law $\lambda_{2t}$ of $U_{2t}$. Hence, by Theorem \ref{bian} 
the density of $\nu_t$ is given by the formula $\kappa_t(\omega)=\rho_{2t}(\omega)$ and the support is the full unit circle for $t>2$ and the set $\{e^{i\theta}: |\theta|<g(2t)\}$ for $t\in [0,2]$.

In the rest of the paper, we illustrate how the family of measure $(\nu_t)_{t\geq0}$ provides a continuous interpolation between freeness and different type of independence.

\subsection{Free initial operators}
If $R$ and $S$ are free, then Proposition 2.5 in \cite{Tar}, implies that
\begin{align*}
H(0,z)=\sqrt{1+4z\left(\frac{b^2}{(1-z)^2}-\frac{a^2}{(1+z)^2}\right)}.
\end{align*}
Then it follows from \eqref{defK} that
\begin{align*}
K(0,z)=\sqrt{H(0,z)^2-\left(a\frac{1-z}{1+z}+b\frac{1+z}{1-z} \right)^2}=\sqrt{1-(a+b)^2}.
\end{align*}
But the facts exposed (under the ODEs \eqref{characteristic}) in section 3 show that $K(t,z)=K(0,\phi_t^{-1}(z))$ holds for every $z\in\mathbb{D}$. 
This implies that $K(t,z)=\sqrt{1-(a+b)^2}$ for any $t\geq0$, 
and therefore $\nu_t$ coincides with the measure $\nu_\infty$. 

\subsection{Classically independent initial operators }
In this case, the measure $\nu_t$ is considered as a $t$-free convolution which interpolates between classical independence and free independence (see \cite{Ben-Lev}). 
Let $R,S$ two independent symmetries, from the facts exposed above Lemma 5.4 in \cite{Tar}, we have
\begin{equation*}
 H(0,z)=1+2\sum_{n\geq 1} \tau(R^n)\tau(S^n) z^n=\frac{1+z^2+2z\tau(R)\tau(S) }{1-z^2}.
\end{equation*} 
In particular, when $\tau(R)=\tau(S)=0$, the function $H(t,z)$ satisfies the PDE 
 \begin{equation*}
\partial_tH+zH\partial_zH= 0,\quad H(0,z)=\frac{1+z^2}{1-z^2}
\end{equation*}
and hence,  by \eqref{standardpde}, it coincide with $H_{\lambda_{4t}}(z^2)$.
We retrieve then the result obtained in \cite[Theorem 3.6]{Ben-Lev}: for any $t\geq0$, the push-forward of $\nu_t$ by the map $z\mapsto z^2$ coincide with the law of $U_{4t}$. In particular,
the density of $\nu_t$ is given by $\kappa_t(\omega)=\rho_{4t}(\omega^2)$ for any $\omega$ in the unit circle and the support is the full unit circle for $t>1$ and the set $\{e^{i\theta}: |\theta|<g(4t)/2\}$ for $t\in [0,1]$.

\subsection{Boolean independent initial operators }
To a given probability measure $\mu$ on  the unit circle, we keep the same notations $\psi_{\mu},H_{\mu}$ and $\chi_{\mu}$ as in section 2.  Let $\mu_1,\mu_2\in\mathscr{M}_{\mathbb{T}}$ and set $F_{\mu}(z)=\frac{1}{z}\chi_{\mu}(z)$. Then the multiplicative boolean convolution $\mu=\mu_1 \bmultg \mu_2$ is uniquely
determined  by (see \cite{Tarek15} or \cite{Franz04}  for more details)
\[
F_\mu(z) = F_{\mu_1}(z)F_{\mu_2}(z).
\]
Then, for boolean independent symmetries  $R,S$ with law $\mu=\frac{\delta_1+\delta_{-1}}{2}$, we have 
\[
\psi_{\mu}(z)=\frac{z^2}{1-z^2},\quad \chi_{\mu}(z)=z^2,\quad  F_{\mu}(z)=z
\]
and therefore $F_{\mu \bmultg \mu}(z)=F_{\mu}(z)^2=z^2$.
It follows that
\[
\psi_{\mu \bmultg \mu}(z)=\frac{z^3}{1-z^3}\quad {\rm and}\ H_{\mu \bmultg \mu}(z)=\frac{1+z^3}{1-z^3}.
\]
Hence, by \eqref{standardpde} the Herglotz transform $H(t,z)$ of $\nu_t$ and $H_{\lambda_{6t}}(z^3)$ solve the same PDE with the initial condition $H(0,z)=(1+z^3)/(1-z^3)$. By uniqueness, it follows that
 the push-forward of $\nu_t$ by the map $z\mapsto z^3$ coincide with the law of $U_{6t}$, for any $t\geq0$. In particular, we have $\kappa_t(\omega)=\rho_{6t}(\omega^3)$ for any $\omega$ in the unit circle and $\nu_t$ is supported in the full unit circle for $t>2/3$ and the set $\{e^{i\theta}: |\theta|<g(6t)/3\}$ for $t\in [0,2/3]$.

\subsection{Monotone independent initial operators }

For $\mu_1,\mu_2\in\mathscr{M}_{\mathbb{T}}$, the multiplicative monotone convolution $\mu=\mu_1 \triangleright \mu_2$ is uniquely
determined  by (see \cite{Tarek15} or \cite{Franz05}  for more details)
\[
\chi_\mu(z) = \chi_{\mu_1}\big(\chi_{\mu_2}(z)\big).
\]
Here, we shall compute the measure $\nu_t$  for monotone independent symmetries  $R,S$ with law $\mu=\frac{\delta_1+\delta_{-1}}{2}$. As usual, we have
\[
\psi_{\mu}(z)=\frac{z^2}{1-z^2},\quad \chi_{\mu}(z)=z^2,
\]
and then $\chi_{\mu \triangleright \mu}(z)=\chi_{\mu}\big(\chi_{\mu}(z)\big)=z^4$.
Hence,
\[
\psi_{\mu \triangleright \mu}(z)=\frac{z^4}{1-z^4}\quad {\rm and}\ H_{\mu \triangleright \mu}(z)=\frac{1+z^4}{1-z^4}.
\]
It follows that $H(t,z)=H_{\lambda_{8t}}(z^4)$ by uniqueness. Thus,
 the push-forward of $\nu_t$ by the map $z\mapsto z^4$ coincide with the law of $U_{8t}$, for any $t\geq0$. In particular, we have $\kappa_t(\omega)=\rho_{8t}(\omega^4)$ for any $\omega$ in the unit circle and $\nu_t$ is supported in the full unit circle for $t>1/2$ and the set $\{e^{i\theta}: |\theta|<g(8t)/4\}$ for $t\in [0,1/2]$.

Finally, we remind (see the section 5.1) that $\nu_t=\nu_0\boxtimes\lambda_{2t}$  for centered initial operators  $R,S$ (i.e. $\tau(R)=\tau(S)=0$). Hence, the discussions so far can be summarized in the Theorem \ref{convol}.






\end{document}